\documentclass[10pt,a4paper,notitlepage,headinclude,twoside]{article}
\usepackage[utf8]{inputenc}
\usepackage[british]{babel}
\usepackage{xargs}                      
\usepackage{cite}
\usepackage[pdftex,dvipsnames]{xcolor}  

\usepackage{blindtext}   
\usepackage{geometry}
\usepackage{amsmath, amsthm, amsfonts, amssymb}
\usepackage[alphabetic]{amsrefs}   
\usepackage{mdframed}
\usepackage{bm}   
\usepackage{slashed}     
\usepackage{xcolor}    
\usepackage[pdftex]{graphicx}
\usepackage[all]{xy}  
\usepackage{ifthen}
\usepackage{tikz}
\usepackage{sseq}
\usepackage{index}
\usepackage{epigraph} 
\usepackage[colorinlistoftodos,prependcaption,textsize=tiny]{todonotes} 
\usepackage[english, refpage]{nomencl}                    
\usetikzlibrary{matrix,arrows}

\usepackage{nameref}
\makeatletter
\newcommand*{\currentname}{\@currentlabelname}
\makeatother
\usepackage{fancyhdr}  
\usepackage{wrapfig}   

\usepackage{hyperref}  
\hypersetup{colorlinks=true, linkcolor=blue, citecolor=blue, urlcolor=blue,}

\theoremstyle{plain} 
\newtheorem{proposition}{Proposition}[section]

\newtheorem{lemma}[proposition]{Lemma}
\newtheorem{cor}[proposition]{Corollary}

\newtheorem{thmx}{Theorem} 
\newtheorem{lemx}[thmx]{Lemma}

\theoremstyle{definition}

\theoremstyle{remark}

\newcommandx{\unsure}[2][1=]{\todo[linecolor=red,backgroundcolor=red!25,bordercolor=red,#1]{#2}}
\newcommandx{\change}[2][1=]{\todo[linecolor=blue,backgroundcolor=blue!25,bordercolor=blue,#1]{#2}}
\newcommandx{\info}[2][1=]{\todo[linecolor=OliveGreen,backgroundcolor=OliveGreen!25,bordercolor=OliveGreen,#1]{#2}}
\newcommandx{\improvement}[2][1=]{\todo[linecolor=Plum,backgroundcolor=Plum!25,bordercolor=Plum,#1]{#2}}
\newcommandx{\thiswillnotshow}[2][1=]{\todo[disable,#1]{#2}}

\newcommand{\pr}{\mathord{\mathrm{pr}}}                 

\newcommand{\N}{\mathord{\mathbb{N}}}                   
\newcommand{\Z}{\mathord{\mathbb{Z}}}                   
\newcommand{\Q}{\mathord{\mathbb{Q}}}                   
\newcommand{\R}{\mathord{\mathbb{R}}}                   

\newcommand{\CP}{\mathord{\mathbb{C}P}}
\newcommand{\id}{\mathord{\mathrm{id}}}
\newcommand{\placeholder}{\text{-}}
\newcommand{\const}{\mathrm{const}}

\newcommand{\SpecOrth}{\mathbf{Sp}^\mathrm{orth}}   

\newcommand{\CFOri}{\mathrm{CF}}

\newcommand{\deff}{\mathrel{\vcenter{\offinterlineskip\hbox{.}\vskip-.80ex\hbox{.}}}\joinrel \hskip 1pt =} 

\usepackage[T1]{fontenc}

\makenomenclature
\makeindex

\title{Line Bundle Twists for Unitary Bordism are Ghosts}
\author{Thorsten Hertl}
\date{}

\begin{document}
\maketitle
\begin{abstract}
    We prove that the canonical twist $\zeta \colon K(\Z,3) \rightarrow BGL_1(MSpin^c)$ does not extend to a twist for unitary bordism by showing that every continuous map $f \colon K(\Z,3) \rightarrow BGL_1(MU)$ loops to a null homotopic map.
\end{abstract}
\section{Introduction}

Twisted cohomology, originally invented by Steenrod \cite{steenrod1943homology}, allows to classify geometric objects that cannot be classified by untwisted cohomology without further orientation conditions. 

While Steenrod defined twisted (co-)homology algebraically in terms of chain complexes with local coefficients, a modern approach to twisted (co)homology  in terms of twisted spectra is provided by the work of \cite{ando2014categorical}, \cite{ando2010twists}, and \cite{Ando_Units_2014}.
Their construction goes roughly as follows: An $A_\infty$-ring spectrum $R$ has a space of units $GL_1(R)$, which deloops to a classifying space $BGL_1(R)$.
A \emph{twist} is a map $\xi \colon X \rightarrow BGL_1(R)$ from which we can construct a Thom spectrum $X^\xi$, whose homotopy groups are the $\xi$-twisted homology groups of $X$.

Our motivation for the work presented here arose from the results of \cite{hebestreit2020twisted}, see also \cite{hebestreit2020homotopical}. 
The authors construct point set models for twisted $Spin^c$ bordism and twisted $K$-theory over $K(\Z,3)$ as well as a model for a twisted Atiyah-Bott-Shapiro orientation $\alpha_{ABS}\colon MSpin^c_{K(\Z,3)} \rightarrow K_{K(\Z,3)}$.
In  \cite{hebestreit2020twisted}*{App. C} it is shown that the two approaches agree, so that the point set model also arises from a twist map $\zeta \colon K(\Z,3) \rightarrow BGL_1(MSpin^c)$.

One of its application concerns the bordism-determines-homology question.
The classical results of \cite{ConnerFloyd_CobVsK_1966} and \cite{hopkins1992spin} state that $K$-theory is completely described by the Conner-Floyd orientation $\mathrm{CF} \colon MU \rightarrow K$ and the ABS orientation $\alpha_{ABS} \colon MSpin^c \rightarrow K$ in the sense that
\begin{equation*}
    MU_\ast(X) \otimes_{MU_\ast(\mathrm{pt})} K_\ast(\mathrm{pt}) \cong K_\ast(X) \cong MSpin^c_\ast(X) \otimes_{MSpin^c_\ast(\mathrm{pt})} K_\ast(\mathrm{pt})
\end{equation*}
for all spectra $X$.
In \cite{baum20XXtwisted} the latter isomorphism is proved by geometric means in the twisted set-up using the $K(\Z,3)$-twisted ABS orientation.

It is natural to ask whether one can also extend the first isomorphism to the $K(\Z,3)$-twisted setup such that it reduces to the isomorphism of \cite{baum20XXtwisted} under the canonical map $MU \rightarrow MSpin^c$.
The first step would be to construct a twisted unitary bordism spectrum over $K(\Z,3)$ and a comparison map $MU_{K(\Z,3)} \rightarrow MSpin^c_{K(\Z,3)}$ that extends the twisted Atiyah-Bott-Shapiro orientation 'to the left'.
The first result of this article shows that this is impossible.

\begin{thmx}\label{Theorem_A}
  The twist map $\zeta \colon K(\Z,3) \rightarrow BGL_1(MSpin^c)$ does not factor through unitary bordism.
  More precisely, there are no continuous maps $S$ and $T$ that make the following diagram homotopy commutative
  \begin{equation*}
      \xymatrix{K(\Z,3) \ar[r]^-\zeta \ar[rd]_T & BGL_1(MSpin^c) \\
      & BGL_1(MU). \ar[u]_{S} }
  \end{equation*}
\end{thmx}

Theorem \autoref{Theorem_A} is, in fact, a consequence of the following purely homotopy-theoretical result.
\begin{thmx}\label{Theorem_B}
  Each continuous map $f \colon K(\Z,3) \rightarrow BGL_1(MU)$ loops to a null homotopic map $\Omega f \colon K(\Z,2) \rightarrow GL_1(MU)$.
\end{thmx}
Its (slightly more general) stable counterpart reads as follows.
\begin{lemx}\label{Lemma_C}
  Each map of spectra $\Sigma_+^\infty K(\Z,2) \rightarrow MU$ that is a ring map up to homotopy induces a homomorphism between the Pontrjagin rings $MU_\ast(K(\Z,2))$ and $MU_\ast(MU)$ that vanishes on $\widetilde{MU}_\ast(K(\Z,2))$. 
\end{lemx}

Partial results strongly indicating Theorem \ref{Theorem_A} have been obtained by Joel Meier in his unpublished Master's Thesis \cite{joelmeier2021master}. 
He showed that there is no $H$-map $K(\Z,2) \rightarrow SO/U$ whose composition with the fibre comparison map $SO/U \rightarrow K(\Z,2) = \mathrm{hofib}(BSpin^c \rightarrow BSO)$ is homotopic to the identity.
Such a map, when it deloops, would produce a non-trivial map $T$ in Theorem \ref{Theorem_A}.
In this case, $S$ would arise from a map of spectra, a condition we do not assume. \\ $ $

Lemma \ref{Lemma_C} has a purely algebraic application.
A \emph{formal group law} over a ring $R$ is a formal power series $F \in R[[x,y]]$ that satisfies the same condition as the Taylor series expansion of the multiplication of a one dimensional Lie group.
Let $L$ be the Lazard ring, the ring over which the universal formal group law is defined.
Lemma \ref{Lemma_C} and Quillen's result \cite{Quillen_FormalGroupLaws_1969} that $L$ is isomorphic to $MU_\ast(\mathrm{pt})$ imply

\begin{thmx}\label{Theorem_D}
  There is no non-trivial homomorphism between the universal formal group law $F_{\mathrm{univ}}$ and the multiplicative formal group $F_{\mathrm{mult}}(x,y) \deff x + y + xy$ on $L$.
  More precisely, if $g \in L[[x]]$ is a formal power series with $g(0) = 0$ and
  \begin{equation*}
      g(F_{\mathrm{univ}}(x,y)) = F_{\mathrm{mult}}(g(x),g(y)),
  \end{equation*}
  then $g=0$.
\end{thmx}

\textbf{Outline of the Proof:}
 Theorem \ref{Theorem_A} follows immediately from Theorem \ref{Theorem_B} once we know that $\zeta$ is sufficiently non-trivial.
 We will translate the statement of Theorem \ref{Theorem_B} into a stable homotopy theoretical one and use the Adams-Novikov spectral sequence to deduce it from Lemma \ref{Lemma_C}.
 The main task is the proof of Lemma \ref{Lemma_C}, which is purely algebraic.
 We will show that $MU_\ast(\mathrm{pt})$-algebra homomorphisms between the Pontrjagin rings $MU_\ast(K(\Z,2))$ and $MU_\ast (MU)$ are uniquely determined by the image of a single element in low degrees.
 Algebraic arguments then imply that any algebra homomorphism $MU_\ast(\CP^\infty) \rightarrow MU_\ast(MU)$ factors through the ground ring $MU_\ast(\mathrm{pt})$.
 Using the natural homology co-operations, we derive that no such algebra homomorphism can come from a map of spectra.

\textbf{Organisation of the Paper:}
We first give a short summary of the theory of orthogonal spectra and set up conventions in Section 2. 
Section 3 is devoted to the proof of Theorem \ref{Theorem_B} and Lemma \ref{Lemma_C}.
The proof of Theorem \ref{Theorem_A} will be carried out in Section 4.
Section 5, in which we prove Theorem \ref{Theorem_D} is logically independent of Section 4.

\textbf{Acknowledgements:} 
This question arose from the unpublished Master's Thesis of Joel Meier \cite{joelmeier2021master}. I would like to thank him for several discussions and reading an earlier draft of this article. I furthermore thank Thorben Kastenholz for an illuminating example and my two PhD advisors, Thomas Schick and Wolfgang Steimle, for their valuable advice.

\section{Foundations and Conventions}

Throughout the entire paper we work in the category of compactly generated Hausdorff spaces.
We assume that the reader is familiar with the notion of orthogonal spectra at the level of \cite{Roitz}.
The facts important to us are that orthogonal ring spectra form a symmetric monoidal model category \cite{Roitz}*{Theoremm 6.4.8} with the sphere spectrum $\mathbb{S}$ as unit.
The monoid objects are precisely (unital) ring spectra.
The category of ring spectra also carries a model structure such that forgetting the ring structure is a Quillen right adjoint \cite{Roitz}*{Theorem 6.6.12}.
Suspension and evaluation at the inital space form a strong symmetric monoidal adjunction
\begin{equation}\label{Susp Ev - Adjunction}
    \xymatrix{\Sigma_+^\infty \colon \mathbf{Top} \ar@<0.5ex>[rr] && \mathbf{Top}_\ast \ar@<0.5ex>[rr]^-{\Sigma^\infty} \ar@<0.5ex>[ll] && \SpecOrth \ar@<0.5ex>[ll]^-{\mathrm{Ev}_0} :\mathrm{Ev}_0,}
\end{equation}
see \cite{Roitz}*{Lemma 6.3.21} and \cite{Roitz}*{Lemma 6.3.18}. 
Consequently, the suspension spectrum of a topological monoid is a ring spectrum and the subspace $R_0^\times$ of all elements of the inital space $R_0$ of a ring spectrum $R$ that induce invertible elements in $\pi_0(R)$ is a topological monoid.
This correspondence is also valid for morphisms.

Adjunction (\ref{Susp Ev - Adjunction}) maps the canonical cylinder objects $(\placeholder)\times [0,1]$ and $(\placeholder) \wedge [0,1]_+$ to each other so that $H$-spaces give rise to naive ring spectra\footnote{that are spectra equipped with a multiplication and a unit map that are associative and unital only up to homotopy} and $H$-maps give rise to naive ring homomorphisms.

If $R$ is a cofibrant and fibrant object in the model category of orthogonal ring spectra, then its \emph{its space of units} $GL_1(R)$ is the set of all path components of $R_0 = \Omega^\infty R$ that represent invertible elements in $\pi_0(R) = \pi_0(R_0)$.
It is a group like topological monoid.

For a topological monoid $G$, we denote with $|G_\bullet|$ the geometric realisation of the singular simplicial set $S_\bullet(G)$.
It is again a topological monoid and the co-unit $\varepsilon \colon |G_\bullet| \rightarrow G$ is a continuous monoid-homomorphism and a weak homotopy equivalence.

For a group like topolgical monoid, we define its classifying space $BG \deff |B_\bullet(e,|G_\bullet|,e)|$ to be the (thin) geometric realisation of the bar construction of $|G_\bullet|$ in sense of May, \cite{may1975classifying}*{§11}.
It comes with a quasi fibration $EG \rightarrow BG$, where $EG$ is a pointed, contractible, free $|G_\bullet|$-right space.
In this setup, Sugawara \cite{sugawara1957condition} constructed a \emph{group completion map} $\iota \colon BG \rightarrow \Omega BG$, which is an $H$-map \cite{sugawara1957condition}*{Lemma 11} and hence a weak homotopy equivalence (as $G$ is group like).
\section{Proof of Theorem \ref{Theorem_B}}

We first prove that Theorem \ref{Theorem_B} under the assumption of Lemma \ref{Lemma_C}.
To this end, we need good models for the involved players so that the space of units are strict monoids.

A model for $K(\Z,2) \simeq BU(1)$ is given by $\CP^\infty$. 
The $H$-structure induced by the group structures is modeled by polynomial multiplication 
\begin{align*}
    \mu \colon \bigl([\zeta_0:\zeta_1:\dots],[\omega_0:\omega_1:\dots]\bigr) \mapsto \bigl([\zeta_0\omega_0:\zeta_1\omega_0 + \zeta_0\omega_1:\zeta_2\omega_0 + \zeta_1\omega_1 + \zeta_0\omega_2:\dots]\bigr).
\end{align*}
It turns $\CP^\infty$ into a group-like monoid, with complex conjugation as homotopy inverse.

Pick an orthogonal ring spectrum that represents complex bordism. 
We denote its cofibrant-fibrant replacement in the model category of ring spectra with $MU$.

\begin{lemma}\label{StableReduction - Lemma}
 Adjunction (\ref{Susp Ev - Adjunction}) induces an injection 
 \begin{equation*}
     [\CP^\infty, GL_1 MU] \hookrightarrow [\Sigma_+^\infty \CP^\infty; MU]
 \end{equation*}
 and the image of the constant maps are precisely those homotopy classes that induce the zero map on $\widetilde{MU}_\ast(\CP^\infty)$.
\end{lemma}
\begin{proof}
 Since $GL_1(MU) \subseteq MU_0 = \Omega^\infty MU$ is a collection of path components, this inclusion induces an injective map on homotopy classes.
 As adjunction (\ref{Susp Ev - Adjunction}) is a Quillen adjunctin, we get an injective map
 \begin{equation*}
     \xymatrix{[\CP^\infty,GL_1(MU)] \ar@{^{(}->}[r] & [\CP^\infty,MU_0] \ar[r]^-{Ad}_-{\cong} & [\Sigma_+^\infty \CP^\infty;MU],}
 \end{equation*}
 
 It follows from the Adams-Novikov spectral sequence that the Hurewicz homomorphism 
 \begin{align*}
  [\Sigma_+^\infty \CP^\infty; MU] \rightarrow \mathrm{Hom}_{MU_\ast MU}(  MU_\ast\CP^\infty;  MU_\ast MU), 
 \end{align*}
 whose target is the set of all $MU_\ast MU$-comodule maps, is bijective.
 Indeed, as described in \cite{Switzer}*{Chapter 19}, the edge homomorphism of the ANSS agrees with the Hurewicz homomorphism. 
 Since $MU_\ast MU$ is a free $MU_\ast MU$-comodule and $MU_\ast(\CP^\infty)$ is a free $MU_\ast(\mathrm{pt}$)-module \cite{Switzer}*{Prop 16.30}, the ANSS is concentrated in the column 
 \begin{equation*}
    E^2_{0,t} = \mathrm{Hom}_{MU_\ast MU}^t(  MU_\ast\CP^\infty;  MU_\ast MU),    
 \end{equation*}
 by \cite{Switzer}*{Prop 19.7}, which forces the spectral sequence to collapse for algebraic reasons.
 As the spectral sequence converges to $[\Sigma_+^\infty \CP^\infty; MU]_\ast$, the edge homomorphism must be an isomorphism. 
 
 The constant maps 
 \begin{equation*}
   \xymatrix{\CP^\infty \ar[rr]^-{\mathrm{const}} & & \mathrm{pt} \ar[rr]^-c & & MU_0 = \Omega^\infty MU}
 \end{equation*}  
 correspond under the adjunction to the following maps of spectra
 \begin{equation*}
  \xymatrix{\Sigma_+^\infty \CP^\infty \ar[rr]^{\Sigma_+^\infty\mathrm{const}} & & \Sigma_+^\infty \mathrm{pt} = \mathbb{S} \ar[rr]^{Ad(c)} & & MU,} 
 \end{equation*}  
 so they induce the zero map on $\widetilde{MU}_\ast(\CP^\infty) = MU_\ast(\CP^\infty,\mathrm{pt}) = \ker \const_\ast$. 
 The bijection of the Hurewicz map conversely shows that every $MU_\ast(MU)$ comodule map that vanishes on $\widetilde{MU}_\ast(\CP^\infty) \subseteq MU_\ast(\CP^\infty)$ must come from a map of spectra that is homotopic to (the adjoint of) a constant map.
  This gives the result.
\end{proof}

\begin{proof}[Proof of Theorem \ref{Theorem_B}]
 Let $f \colon B\CP^\infty \rightarrow BGL_1(MU)$ be a twist and $\Omega f$ its induced map on loop spaces.
 The monoids $\CP^\infty$ and $GL_1(MU)$ are group-like, so the group completion maps are weakly homotopy equivalent $H$-maps.
 Since post-composition with weak equivalences induce bijections between homotopy classes, see \cite{Switzer}*{Thm. 6.31}, there are unique homotopy classes that make the following diagram commute up to homotopy
 \begin{equation*}
     \xymatrix{   \CP^\infty \ar@{-->}[rr]  && GL_1(MU)   \\ 
     |\CP^\infty_\bullet| \ar[u]^{\varepsilon}_\simeq \ar[d]_\iota^\simeq  \ar@{-->}[rr] && |GL_1(MU)_\bullet| \ar[u]^\simeq_\varepsilon \ar[d]^\iota_\simeq  \\
       \Omega B\CP^\infty \ar[rr]^{\Omega f} && \Omega BGL_1(MU). }
 \end{equation*}
 The maps of the group-competition zig-zags and $\Omega f$ are $H$-maps, so the homotopy class of the lift $\CP^\infty \dashrightarrow GL_1(MU)$ must consists of $H$-maps.
 We denote any representative of the lifted homotopy class again with $\Omega f$.
 
 Under the adjunction (\ref{Susp Ev - Adjunction}) the $H$-map $\Omega f$ corresponds to a ring homomorphism up to homotopy
 \begin{equation*}
     \varphi \deff Ad(\Omega f) \colon \Sigma_+^\infty \CP^\infty \rightarrow MU.
 \end{equation*}
 It follows from Lemma \ref{StableReduction - Lemma} that $\Omega f$ is null-homotopic if $\varphi$ is.
 The latter follows from the second half of Lemma \ref{StableReduction - Lemma} and Lemma \ref{Lemma_C}.
\end{proof}

  It remains to prove Lemma \ref{Lemma_C}, which is reformulated in Lemma \ref{Lemma_RingMapsAreZero} below.
  Before we carry out the proof, we recall and develop the required structural results.
  
  We recall from Switzer \cite{Switzer}*{Prop. 16.29} that 
  $MU^\ast(\CP^\infty) \cong MU^\ast[[c]]$,
  where $c \in \widetilde{MU}^2(\CP^\infty)$ is the universal first Chern-class.
  The homology is isomorphic to a free $MU_\ast(\mathrm{pt})$-module. 
  More precisely, we have $MU_\ast(\CP^\infty) \cong MU_\ast(\mathrm{pt})\{1,\beta_1,\beta_2,\dots\}$, where $\beta_j$ is dual to $c^j$,  see \cite{Switzer}*{Prop 16.30}.
  Since $MU_\ast(\mathrm{pt})$ is a polynomial ring over $\Z$, see \cite{Switzer}*{Theorem 20.25}, it is torsion free.
  We recall further that $(MU_\ast MU, MU_\ast(\mathrm{pt}))$ is a Hopf-algebroid.
  This means, in particular, that the multiplication of $MU$ turns $MU_\ast MU$ into a commutative ring, that the diagonal map $\Delta \colon MU \rightarrow MU \wedge MU$ induces a co-action on $MU_\ast MU$, and that these two structures are compatible. 
  However, $MU_\ast(\mathrm{pt})$ acts canonically from the left and right on $MU_\ast(MU)$ and these actions are not the same.
  There are therefore a left unit $\eta_L$ and a right unit $\eta_R$ on $MU_\ast MU$.
  The left unit sends $\lambda \in MU_\ast(\mathrm{pt})$ to $\lambda \cdot 1$, where $1$ is the unit of the Pontrjagin ring.
  For more details, see \cite{Switzer}*{p. 414 ff} or \cite{Kochman}*{Prop. 4.5.3}.
  We will usually consider $MU_\ast MU$ as a left-module over $MU_\ast(\mathrm{pt})$.
  As an $MU_\ast(\mathrm{pt})$-algebra, on which the ring $MU_\ast(\mathrm{pt})$ acts from the left, $MU_\ast (MU)$ is a polynomial ring:  
  \begin{equation*}
   MU_\ast(MU) \cong MU_\ast(\mathrm{pt})[b_1,b_2,\dots],
  \end{equation*}    
   where $b_j \in MU_{2j}(MU)$, see \cite{Switzer}*{Theorem 17.16} or \cite{Kochman}*{Prop. 4.4.4}.
  
  \begin{lemma}\label{Lemma_MUProduct}
   The additive generators $\{1, \beta_1, \beta_2, \dots\}$ of the Pontrjagin ring $MU_\ast(\CP^\infty)$ satisfy the multiplicative relations:
   \begin{equation*}
     \beta_i \bullet \beta_j \deff MU_\ast(\mu)(\beta_i \otimes \beta_j) = \sum_{k\geq 0} (\sum_{\substack{a_1 + \dots + a_k = i \\ b_1 + \dots + b_k = j}} \prod_{r=1}^k \alpha_{a_r b_r} ) \beta_k,
   \end{equation*}
   where the $\alpha_{ij}\in MU_{2(i+j-1)}$ are the coefficients in the unitary bordism formal group law.
  \end{lemma}
  \begin{proof}
   Recall that $MU^\ast(\CP^\infty) = MU^\ast(\mathrm{pt})[[c]]$ so that $MU^\ast(\CP^\infty \times \CP^\infty) = MU^\ast(\mathrm{pt})[[c,d]]$ are rings of formal power series over $MU^\ast(\mathrm{pt})$.
   The monoid structure $\mu$ induces under this identification a formal group law \cite{Kochman}. 
   More precisely
   \begin{equation*}
    MU^\ast(\mu)(c) = \sum_{i,j\geq 0} \alpha_{i,j} c^i d^j, 
   \end{equation*}    
   where $\alpha_{ij} \in MU^{-2(i+j-1)}(\mathrm{pt)} = MU_{2(i+j-1)}(\mathrm{pt})$.
   In particular, $\alpha_{10} = \alpha_{01} = 1$ and $\alpha_{i0}=\alpha_{0i} = 0$ if $i\neq 1$, see \cite{Kochman}*{p.148}.
   
   Since $MU^\ast(\mu)$ is a ring homomorphism, the images of the higher powers are given by
   \begin{equation*}
    MU^\ast(\mu)(c^k) = \left(MU^\ast(\mu)(c))^k\right) = \sum_{i,j\geq 0} \left(\sum_{\substack{i_1 + \dots + i_k = i \\ j_1 + \dots + j_k = j}} \prod_{r=1}^k \alpha_{i_r j_r} \right) c^i d^j.
   \end{equation*}
   Under the Künneth isomorphism $MU_\ast(\CP^\infty \times \CP^\infty) \cong MU_\ast(\CP^\infty) \otimes_{MU_\ast(\mathrm{pt})} MU_\ast(\CP^\infty)$, the dual element of $c^id^j$ is precisely $\beta_i \otimes \beta_j$, see \cite[Prop 4.3.2]{Kochman} for a reference.
   If $\langle \cdot , \cdot \rangle$ denotes the natural pairing between homology and cohomology, then we have
   \begin{align*}
    \langle MU_\ast(\mu)(\beta_i \otimes \beta_j), c^k\rangle &= \langle \beta_i \otimes \beta_j, MU^\ast(\mu)(c^k) \rangle   \\
    &= \langle \beta_i \otimes \beta_j , \sum_{a,b\geq 0} \sum_{\substack{a_1 + \dots + a_k = a \\ b_1 + \dots + b_k = b}} \prod_{r=1}^k \alpha_{a_r b_r}  c^a d^b \rangle \\
    &= \sum_{\substack{a_1 + \dots + a_k = i \\ b_1 + \dots + b_k = j}} \prod_{r=1}^k \alpha_{a_r b_r} .
   \end{align*}     
  Since $\beta_k$ is dual to $c^k$ we end up with the claimed formula.
  \end{proof}
  
  To make calculations manageable, we eventually carry them out in complex $K$-theory.
  Let us therefore recall the connection between complex bordism and complex $K$-theory following \cite[p.423 ff and 433 ff]{Switzer}.
  If $t \in K_2(\mathrm{pt}) = K^{-2}(\mathrm{pt})$ denotes the Bott-element, then Bott periodicity yields $K_\ast(\mathrm{pt}) \cong \Z [t,t^{-1}]$.
  The Pontrjagin ring $K_\ast(K)$ embeds into $K_\ast(K) \otimes \Q \cong \Q[u,v,u^{-1},v^{-1}]$.
  It is a Hopf-algebroid with left unit $\eta_L(t) = u$ and right unit $\eta_R(t) = v$.
  We do not need its co-action here.
 
 As in the case of complex bordism, $K_\ast(\CP^\infty)$ is a free left $K_\ast(\mathrm{pt})$-module and $K_\ast(MU)$ is isomorphic to a polynomial ring over $K_\ast(\mathrm{pt})$. 
 More precisely, we have an isomorphism of left $K_\ast(\mathrm{pt})$-modules
 \begin{equation*}
  K_\ast(\CP^\infty) \cong K_\ast(\mathrm{pt})\{1,t^1Y_1,t^2Y_2,\dots\}
 \end{equation*}
 with $Y_j \in K_0(\CP^\infty)$ and an isomorphism of $K_\ast(\mathrm{pt})$-algebras
 \begin{equation*}
  K_\ast(MU) \cong K_\ast(\mathrm{pt})[b_1^K,b_2^K,\dots],
 \end{equation*}
 where the underlying ring $K_\ast(\mathrm{pt})$ acts from the left.
 In \cite{Switzer} the elements $b_j^K$ are denoted by $Y_j'$.
 
 The Conner-Floyd orientation is a map of spectra $\CFOri \colon MU \rightarrow K$ that classifies the Thom class of the universal complex vector bundle, see \cite[p.434]{Switzer} or \cite[p.29]{ConnerFloyd_CobVsK_1966}.
 It induces a natural transformation of homology theories $MU_\ast \rightarrow K_\ast$ that maps $\beta_j$ to $t^jY_j$ and $b_j$ to $b_j^K$, respectively \cite[p.434]{Switzer}.
 It follows from the geometric description that $\CFOri$ maps the generator $[\CP^1] \in \Omega_2^U(\mathrm{pt})$ to the Bott element $t \in K_2(\mathrm{pt})$. 
 Alternatively, it follows from \cite[Cor 6.5]{ConnerFloyd_CobVsK_1966} and the fact that the Todd-genus of $\CP^1$ is $1$.
 
 Analogously to Lemma \ref{Lemma_MUProduct} we prove Lemma \ref{Lemma_KProduct}.
 \begin{lemma}\label{Lemma_KProduct}
  The additive generators $\{1,tY_1,t^2Y_2,\dots\}$ of the Pontrjagin ring $K_\ast(\CP^\infty)$ satisfy the multiplicative relations:
  \begin{align*}
   t^iY_i \bullet t^jY_j = t^{i+j} \cdot \sum_{k=\mathrm{max}\{i,j\}}^{i+j} \binom{k}{2k-(i+j)}\binom{2k-(i+j)}{k-j} Y_k.
  \end{align*}
 \end{lemma}
 \begin{proof}
  Let $\tau \rightarrow \CP^\infty$ the tautological line bundle.
  Since the monoid map $\mu \colon \CP^\infty \times \CP^\infty \rightarrow \CP^\infty$ classifies the tensor product of vector bundles, we have $\mu^\ast(\tau) = \pr_1^\ast \tau \otimes \pr_2^\ast \tau$.
  The (standard) complex orientation of complex $K$-theory is given by 
  \begin{equation*}
   y^K \deff t^{-1} \cdot (\tau - 1) \in \tilde{K}^2(\CP^\infty),
  \end{equation*}
  where $1=\varepsilon^1$ denotes the trivial complex line bundle, which is the unit in $K^\ast(\mathrm{pt})$.
  Since multiplication in $K$-theory is induced by the tensor products, we have under the isomorphism $K^\ast(\CP^\infty \times \CP^\infty) \cong K^\ast(\mathrm{pt})[[x,y]]$ the equality
  \begin{align*}
   K^\ast(\mu)(y^K) &= t^{-1}K^\ast(\mu)(\tau-1) = t^{-1}(\pr_1^\ast \tau  \otimes \pr_2^\ast \tau - 1 \otimes 1) \\
   &= t^{-1}(tx + 1)(ty + 1) - t^{-1} \\
   &= x + y + t\cdot xy. 
  \end{align*}
  The proof is now analogous to the proof of Lemma \ref{Lemma_MUProduct} but in contrast to the cited lemma we have to make the formula explicit.
  
  Since $K^\ast(\mu)$ is a ring homomorphism, the images of the higher powers are already determined
  \begin{equation*}
      K^\ast(\mu)((y^K)^k) = (x+y+t\cdot xy)^k = \sum_{l=0}^k \sum_{m=0}^l \binom{l}{m} \binom{k}{l} t^{k-l} x^{k-l+m}y^{k-m}.
  \end{equation*}
 
 As before, $t^iY_i \otimes t^jY_j$ is dual to $x^iy^j$ so that $\langle t^i Y_i \otimes t^jY_j; x^a y^b\rangle = \delta_{a,i}\delta_{b,j}$. 
 This implies
 \begin{align*}
     \langle t^iY_i \bullet t^jY_j ; y^k\rangle &= \langle K_\ast(\mu)(t^iY_i \otimes t^jY_j;y^k) = \langle t^iY_i \otimes t^jY_j , K^\ast(\mu)(y^k)\rangle \\
     &=\sum_{l=0}^k \sum_{m=0}^l \binom{l}{m} \binom{k}{l} t^{i+j-k} \langle t^{k-l+m}Y_i \otimes t^{k-m} Y_j;  x^{k-l+m}y^{k-m} \rangle \\
     &= \begin{cases} 
     \binom{k}{2k-(i+j)} \binom{2k - (i+j)}{k - j} t^{i+j-k} , &\text{ if } 0 \leq 2k - (i+j) \leq k, \\
     0, & \text{ else.}
     \end{cases}
 \end{align*}
 Clearly, the first condition is equivalent to $\lceil (i+j)/2\rceil \leq k \leq i+j$. 
 As $t^kY_k$ is dual to $y^k$, the previous calculations yield
 \begin{equation*}
     t^iY_i \bullet t^jY_j =  t^{i+j} \cdot \sum_{k=\lceil (i+j)/2\rceil}^{i+j} \binom{k}{2k-(i+j)} \binom{2k - (i+j)}{k - j} Y_k. 
 \end{equation*}
 The latter binomial coefficient is zero if $k - j > 2k -(i+j)$ or $k - i > 2k - (i+j)$.
 Thus the formula can be shortened to the claimed formula.
 \end{proof}
  \begin{lemma}\label{Lemma_RingMapsAreZero}
   Every map of (naive) ring spectra $\varphi \colon \Sigma_+^\infty \CP^\infty \rightarrow MU$ induces a homomorphism that vanishes on $\widetilde{MU}_\ast(\CP^\infty)$.
  \end{lemma}
  \begin{proof}
   Since $\widetilde{MU}_\ast(\CP^\infty) = MU_\ast(\mathrm{pt})\{\beta_1,\beta_2,\dots\} \subseteq MU_\ast(\CP^\infty)$ \cite[Prop. 16.30]{Switzer} it suffices to show that $MU_\ast(\varphi)$ sends all $\beta_n$ to zero.
   We abbreviate $MU_\ast(\varphi)$ with $f$.
   Since $f$ is degree preserving, there are elements $\lambda_j \in MU_j(\mathrm{pt})$ such that 
   \begin{equation*}
    f(\beta_1) = \lambda_2 1 + \lambda_0 b_1 =: Q \in MU_2(MU).
   \end{equation*}
   From Lemma \ref{Lemma_MUProduct} we derive $\beta_1 \bullet \beta_1 = 2 \beta_2 + \alpha_{11}\beta_1$ and
   \begin{equation*}
       \beta_1^{\bullet n} = n! \beta_n + {}_n\zeta_{n-1} \beta_{n-1} \dots + {}_n\zeta_1 \beta_1
   \end{equation*}
   for some ${}_n\zeta_j \in MU_{2(n-j)}(\mathrm{pt})$.
   As $f$ is an algebra homomorphism, we deduce inductively
   \begin{equation*}
      \prod_{k=1}^n k! \cdot f(\beta_n) = \prod_{k=1}^{n-1} k! \cdot Q^n + {}_n\eta_{n-1}Q^{n-1} + \dots + {}_n\eta_{1}Q
   \end{equation*}
   for some ${}_n\eta_j \in MU_{2(n-j)}(\mathrm{pt})$.
   Furthermore, 
   \begin{equation*}
    \prod_{k=1}^n k! \cdot f(\beta_n) \equiv \prod_{k=1}^{n-1} k! \cdot Q^n \equiv \prod_{k=1}^{n-1} k! \cdot \lambda_0^n b_1^n \quad \text{ mod } \mathrm{span}\{1,b_1,\dots,b_1^{n-1}\},   
   \end{equation*}
   which implies that $n! \, | \,\lambda_0^n \in MU_0(\mathrm{pt}) \cong \Z$ for all $n \in \N$, so $\lambda_0$ must be zero.
   Thus, $Q = \lambda_2 \cdot 1 \in MU_2(\mathrm{pt})$.
   
   To show that $\lambda_2 = 0$, we consider the image of $Q$ in $K$-theory.
   The following diagram commutes
   \begin{equation*}
    \xymatrix{MU_\ast(\CP^\infty) \ar[rr]^-{MU_\ast(\varphi)} \ar[d]^{\CFOri} & & MU_\ast MU \ar[d]^{\CFOri} \\
    K_\ast(\CP^\infty) \ar[rr]^-{K_\ast(\varphi)} & & K_\ast(MU).}
   \end{equation*}
   Set $g \deff K_\ast(\varphi)$.
   Since $\CFOri$ is an isomorphism in degree $2$, it follows from the previous calculation that
   \begin{align*}
    g(tY_1) = g(\CFOri_2(\beta_1)) &= \CFOri_2(f(\beta_1)) = \CFOri_2(\lambda_2 1_{MU_\ast MU}) \\
    &=\CFOri_2(\lambda_2) 1_{K_\ast(MU)} = \lambda_2 t \cdot 1_{K_\ast MU},
   \end{align*}
   where we abuse notation and denote with $\lambda_2$ an element in $MU_2(\mathrm{pt})$ and the multiple $\lambda_2 \in \Z$ of the generator $[\CP^1] \in \Omega_2^U(\mathrm{pt})$ such that $\lambda_2 \cdot [\CP^1] = \lambda_2$.
   It follows from the Pontrjagin ring structure of $K_\ast(\CP^\infty)$ that
   \begin{align*}
    2g(t^2Y_2) &= g(tY_1 \bullet tY_1) - g(t^2Y_1) = g(tY_1)^2 - tg(tY_1) \\
     &= t^2(\lambda_2^2 - \lambda_2).  
   \end{align*}
  The formulas for the co-actions
 \begin{align*}
  \psi_{\CP^\infty} \colon K_\ast(\CP^\infty) &\rightarrow K_\ast(K) \otimes_{K_\ast(\mathrm{pt})}K_\ast(\CP^\infty), \quad \text{ and}\\
  \psi_{MU} \colon K_\ast(MU) &\rightarrow K_\ast(K) \otimes_{K_\ast(\mathrm{pt})}K_\ast(MU)
 \end{align*}  
 can be found in \cite[Prop 17.38]{Switzer}. 
 Since the co-actions are natural with respect to morphisms induced by maps of spectra, we must have 
 \begin{equation*}
  (\id \otimes g)(\psi_{\CP^\infty}(t^2Y_2)) = \psi_{MU}(g(t^2Y_2)).
\end{equation*}  
 The following calculations\footnote{as above, $\lambda_2 \in \Z$ denotes only the multiple of a generator}, in which $\cdot_r$ denotes the right action, 
 \begin{align*}
  (\id \otimes g)(\psi_{\CP^\infty}(t^2Y_2)) &= \sum_{i+j = 2} (P^j)_{2i} \otimes g(t^jY_j) \\
  &= 1 \otimes \frac{t^2}{2}(\lambda_2^2-\lambda_2)\cdot 1_{K_\ast MU} + p_1 \otimes \lambda_2 t \cdot 1_{K_\ast MU} + 0 \otimes 1 \\
  &= 1_{K_\ast K} \cdot_r \frac{t^2}{2}(\lambda_2^2-\lambda_2) \otimes 1_{K_\ast MU} +  \frac{(v-u)}{2} \cdot_r \lambda_2 t \otimes 1_{K_\ast MU} \\
  &= (\lambda_2^2-\lambda_2)\frac{v^2}{2} \otimes 1_{K_\ast MU} + \frac{\lambda_2 (v-u)v}{2} \otimes 1_{K_\ast MU} \\
  &= \frac{1}{2}\left(\lambda^2_2v^2 - \lambda_2 vu\right) \otimes 1_{K_\ast MU}
 \end{align*}
 and 
 \begin{align*}
  2\psi_{MU}(g(t^2Y_2)) &= \psi_{MU}(t^2(\lambda_2^2-\lambda_2)\cdot 1_{K_\ast MU}) \\
  &= t^2(\lambda^2_2 - \lambda_2) \cdot \psi_{MU}(1) \\
  &= t^2(\lambda^2_2-\lambda_2)\cdot 1_{K_\ast K} \otimes 1_{K_\ast MU} \\
  &= (\lambda_2^2 - \lambda_2) u^2 \otimes 1_{K_\ast MU}.
 \end{align*}
 show that these two elements are equal only if $\lambda_2 = 0$.
 Thus, $Q=0$ and so $MU_\ast(\varphi)$ vanishes on $\widetilde{MU}_\ast(\CP^\infty)$ because $MU_\ast(MU)$ has no torsion.
\end{proof}
\section{Proof of Theorem \ref{Theorem_A}}

Theorem \ref{Theorem_A} will follow immediately from Theorem \ref{Theorem_B} once we have shown that the twist map $\zeta \colon K(\Z,3) \rightarrow BGL_1(MSpin^c)$ in Theorem \ref{Theorem_A} induce non-trivial morphisms on homotopy groups. 
According to \cite{hebestreit2020twisted}*{p.52 ff} the twist map  $\zeta \colon K(\Z,3) \rightarrow BGL_1(Spin^c)$ loops to the map $BU(1) = K(\Z,2) \rightarrow GL_1(MSpin^c)$ that is the adjoint of 
\begin{equation*}
    M\iota \colon \Sigma_+^\infty BU(1) \rightarrow MSpin^c,
\end{equation*}
the map between Thom spectra induced by the fibre inclusion $\iota \colon U(1) \rightarrow Spin^c$.

Recall the canonical homomorphism $\xi \colon Spin^c \xrightarrow{2:1} SO \times U(1)$.
The representation that is used in the construction of the universal $Spin^c(n)$ vector bundle factors through $SO(n)$, i.e. we have by definition
\begin{equation*}
    MSpin_n^c = \mathrm{Th}(ESpin^c(n) \times_{\pr_1 \circ \xi} \R^n).
\end{equation*}
Thus $\xi \colon Spin^c(n) \rightarrow SO(n) \times U(1)$ induces maps
\begin{equation*}
    \xymatrix{\mathrm{Th}(ESpin^c(n) \times_{\pr_1 \circ \xi} \R^n) \ar[r] \ar@{=}[dd] & \mathrm{Th}(E(SO(n)\times U(1)) \times_{\pr_1} \R^n) \ar@{=}[d]  \\
    & \mathrm{Th}\bigl((ESO(n) \times_{\mathrm{taut}} \R^n) \times BU(1) \bigr)  \ar@{=}[d]\\
    MSpin^c(n) \ar[r]^{M_n\xi} & MSO(n) \wedge BU(1)_+ .}
\end{equation*}
These maps are compatible with the inclusion and therefore give rise to a map of spectra
\begin{equation*}
    MSpin^c \xrightarrow{M\xi} MSO \wedge BU(1)_+ = MSO \wedge \Sigma_+^\infty BU(1).
\end{equation*}
We will show  that the composition 
\begin{equation*}
   \xymatrix{ \Sigma_+^\infty BU(1) \ar[r]^{M\iota} & MSpin^c \ar[r]^-{M\xi} & MSO \wedge \Sigma_+^\infty BU(1)}
\end{equation*}
induces non-trivial morphisms on rational homology groups (for spectra). 
By the Thom isomorphism, we obtain the following commutative diagram
\begin{align*}
    \xymatrix{H_\ast(\Sigma_+^\infty BU(1);\Q) \ar[rr]^-{H_\ast(M\xi \circ M\iota)} \ar[dd]_\Phi^\cong && H_\ast(MSO \wedge BU(1)_+; \Q) \ar[d]^\Phi_{\cong} \\
    && H_\ast(BSO \times BU(1);\Q) \ar[d]^{H_\ast(\mathrm{pr}_2)} \\
    H_\ast(BU(1);\Q) \ar[rr]_-{H_\ast(\pr_2 \circ B\xi \circ B\iota)} && H_\ast(BU(1);\Q), }
\end{align*}
where $\Phi$ denotes the Thom isomorphism maps.
Note that $\mathrm{pr}_2 \circ B\xi \circ B\iota$ is homotopic to $B\varphi$, where $\varphi$ denotes the composition of group homomorphisms
\begin{equation*}
    \xymatrix{U(1) \ar[r]^\iota & Spin^c \ar[r]^-\xi & SO \times U(1) \ar[r]^-{\mathrm{pr}_2} & U(1). }
\end{equation*}
The homomorphism $\varphi$ satisfies $\varphi(z) = z^2$.

\begin{lemma}
 The map $B\varphi$ induces non-vanishing homomorphisms on all rational homology groups in even degree.
\end{lemma}
\begin{proof}
Note that $B\varphi$ is an $H$-map that induces the multiplication with $2$ on $H_2(BU(1);\Q)$.
A similar (but easier) computation as in the proof of Lemma \ref{Lemma_KProduct} shows that the Pontrjagin ring structure of $H_\ast(BU(1);\Q) \cong \Q\{1,\gamma_1,\gamma_2,\dots\}$ is given by
\begin{equation*}
    \gamma_i \bullet \gamma_j = \binom{i+j}{i} \gamma_{i+j}.
\end{equation*}
Using this formula, we derive inductively
\begin{equation*}
    H_{2n}(B\varphi;\Q)(\gamma_n) = 2^n \gamma_n,
\end{equation*}
which is obviously non-zero.
\end{proof}
\begin{cor}
 The map $M\iota \colon \Sigma_+^\infty BU(1) \rightarrow MSpin^c$ induce non-trivial homomorphisms on homotopy groups in even degree.
\end{cor}

This Corollary together with Theorem \ref{Theorem_B} immediately implies the main theorem.

\begin{proof}[Proof of Theorem \ref{Theorem_A}]
 Assume that we would have a factorisation up to homotopy
 \begin{equation*}
      \xymatrix{K(\Z,3) \ar[r]^-{\zeta} \ar[rd]_T & BGL_1(MSpin^c) \\
      & BGL_1(MU) \ar[u]_{S} }
 \end{equation*}
 then their loopings were homotopic, too. 
 The adjoint maps then would satisfy
 \begin{equation*}
     M\iota = Ad(\Omega \zeta) \simeq Ad(\Omega S \circ \Omega T).
 \end{equation*}
 But we know from Theorem \ref{Theorem_B} that $\Omega T$, and hence $M\iota = Ad(\Omega S \circ \Omega T)$ too, is null-homotopic. 
 Thus, $\iota$ would be trivial on all homotopy groups, a contradiction.
\end{proof}
\section{Proof of Theorem \ref{Theorem_D}}

Let $R$ be a commutative ring with unit. 
Recall that a (one-dimensional, commutative) \emph{formal group law} over $R$ is a  formal power series $F \in R[[x,y]]$ that satisfies 
\begin{align*}
    F(x,0) = F(0,x) = x, \  F(x,y) = F(y,x), \text{ and } F(x,F(y,z)) = F(F(x,y),z).
\end{align*}
A \emph{homomomorphism} between formal group laws $F$, and $G$ is a formal power series $g \in R[[x]]$ with $g(0) = 0$ and $g(F(x,y)) = G(g(x),g(y))$.

By the work of Lazard, there is a ring $L$ that carries a universal group law $F_{\mathrm{univ}} \in L[[x,y]]$ on it in the sense that, for any other formal group law $F$ over a ring $R$, there is a unique ring homomorphism $\phi \colon L \rightarrow R$ such that $\phi(F_{\mathrm{univ}}(x,y)) = F(x,y)$.
Quillen \cite{Quillen_FormalGroupLaws_1969} identified $L$ as the unitary bordism ring $MU^\ast(\mathrm{pt})$.
Under this identification, the universal group law agrees with the unitary bordism formal group law, that means it is given by the formal power series
\begin{align*}
    MU^\ast(\mu)(c) = \sum_{i,j \geq 0} \alpha_{ij} c^id^j \in MU^\ast[[c,d]] 
\end{align*}
where we use the notation is the same as in Lemma \ref{Lemma_MUProduct}.

This translation into the topological realm allows us to deduce Theorem \ref{Theorem_D} from Lemma \ref{Lemma_C} by showing that these two statements are, in fact, equivalent.

\begin{proof}[Proof of Theorem \ref{Theorem_D}]
 As $MU^\ast(\mathrm{pt})[[c]]_0 = MU^0(\CP^\infty) = [\Sigma_+^\infty \CP^\infty; MU]$, we can represent any formal power series of pure degree zero 
 \begin{align*}
     a = a(c) = \sum_{k \geq 0} a_k c^k, \quad\text{ with } a_k \in MU^{-2k}(\mathrm{pt}),
 \end{align*}
 by a map of spectra $f_a \colon \Sigma_+^\infty \CP^\infty \rightarrow MU$ that is unique up to homotopy.
 The map $f_a$ is, by definition, a naive ring homomorphism if and only if the following diagram commutes up to homotopy
 \begin{equation*}
     \xymatrix{  \Sigma_+^\infty \CP^\infty \wedge \Sigma_+^\infty \CP^\infty \ar[r]^-{f_a \wedge f_a} \ar[d]^{\Sigma_+^\infty \mu} & MU \wedge MU \ar[d]^{m} & \mathbb{S} \ar[d]_{\eta_{\Sigma_+^\infty \CP^\infty}} \ar[dr]^{\eta_{MU}}& \\  \Sigma_+^\infty \CP^\infty \ar[r]^-{f_a} & MU & \Sigma_+^\infty \CP^\infty \ar[r]^{f_a} & MU, } 
 \end{equation*}
 where $m$ denotes the ring structure of $MU$ and $\eta_{(\placeholder)}$ denotes the corresponding unit.
 Translating these diagrams into algebra, this means that
 \begin{equation*}
     \mu^\ast(a) = a(c)\cdot a(d) \qquad \text{ and } \qquad a(0) = 1,
 \end{equation*}
 because $m$ represents the cohomology cross product, see \cite{Switzer}*{p. 270}.
 Quillen's theorem \cite{Quillen_FormalGroupLaws_1969} yields
 \begin{equation*}
     \mu^\ast(a) = \sum_{k\geq 0}a_k \mu^\ast(c)^k = \sum_{k\geq 0} a_k \left(F_{\mathrm{univ}}(c,d)\right)^k = a(F_{\mathrm{univ}}(c,d)).
 \end{equation*}
 Thus, we have $a(F_{\mathrm{univ}}(c,d)) = a(c)\cdot a(d)$ and if we set $g = a-1$, then we get 
 \begin{equation*}
     g(F_{\mathrm{univ}}(c,d)) = g(c) + g(d) + g(c)g(d) = F_{\mathrm{mult}}(g(c),g(d)) \quad \text{ and } \quad g(0)=0.
 \end{equation*}
 In conclusion, an element $g \in MU^\ast(\mathrm{pt})[[c]]$ is a homomorphism between the universal formal group law and the multiplicative formal group law if and only if $g+1$ is represented by a naive ring homomorphism.
 But Lemma \ref{Lemma_C} implies that $g + 1 = 1$ so that $g=0$.
\end{proof}


\bibliographystyle{alpha}
\bibliography{Literatur}

\vspace{1cm}

\small{\textit{E-Mail:} \texttt{thorsten.hertl@stud.uni-goettingen.de}} \\

\small{\scshape{Institut für Mathematik, Universität Göttingen, Bunsenstraße 3-5, D-37037 Göttingen}}

\end{document}